\documentclass[12pt]{amsart}

\usepackage{amsfonts}
\usepackage{amsbsy}
\usepackage{amssymb}
\usepackage{hyperref}

\newcommand{\mb}{\mathbb}
\newcommand{\mf}{\mathfrak}

\renewcommand{\ge}{\geqslant}
\renewcommand{\le}{\leqslant}

\newcommand{\Ker}{\operatorname{Ker}}
\renewcommand{\Im}{\operatorname{Im}}

\newcommand{\codim}{\operatorname{codim}}
\newcommand{\SL}{\operatorname{SL}}
\newcommand{\GL}{\operatorname{GL}}

\newcommand{\Spin}{\operatorname{Spin}}

\newcommand{\diag}{\operatorname{diag}}

\newtheorem{theorem}{Theorem}

\newtheorem{proposition}{Proposition}
\newtheorem{lemma}{Lemma}
\newtheorem{corollary}{Corollary}

\newtheorem*{question*}{Question}

\theoremstyle{definition}

\newtheorem*{dfn*}{Definition}

\theoremstyle{remark}

\newtheorem{remark}{Remark}

\begin{document}

\title[An epimorphic subgroup arising from Roberts' counterexample]{An
epimorphic subgroup arising from\\Roberts' counterexample}

\author{Roman Avdeev}

\thanks{Partially supported by Russian Foundation for Basic Research, grant no. 09-01-00648}

\address{Chair of Higher Algebra, Department of Mechanics and Mathematics,
Moscow State University, 1, Leninskie Gory, Moscow, 119992, Russia}

\email{suselr@yandex.ru}

\date{\today}



\begin{abstract}
In 1994, based on Roberts' counterexample to Hilbert's fourteenth
problem, A'Campo-Neuen constructed an example of a linear action of
a 12-dimen\-si\-onal commutative unipotent group $H_0$ on a
19-dimensional vector space $V$ such that the algebra of invariants
$\Bbbk[V]^{H_0}$ is not finitely generated. We consider a certain
extension $H$ of $H_0$ by a one-dimensional torus and prove that $H$
is epimorphic in $\SL(V)$. In particular, the homogeneous space
$\SL(V)/H$ provides a new example of a homogeneous space with
epimorphic stabilizer that admits no projective embeddings with
small boundary.
\end{abstract}

\maketitle

\section{Introduction}

We work over an algebraically closed field $\Bbbk$ of characteristic
zero. Throughout the paper all topological terms relate to the
Zarisky topology, all groups are supposed to be algebraic and their
subgroups closed.

Let $G$ be a connected affine algebraic group and $H$ a subgroup of
it. We recall that a \emph{projective embedding with small boundary}
of the homogeneous space $G/H$ is an open $G$-equivariant embedding
$\rho\colon G/H \hookrightarrow X$, where $X$ is an irreducible
normal projective $G$-variety and $\codim_X(X \backslash \rho(G/H))
\ge 2$. For a given homogeneous space $G/H$, the existence of such
embedding implies that the algebra $\Bbbk[G/H]$ of regular functions
on $G/H$ consists of constants, that is, $\Bbbk[G/H] = \Bbbk$. A
subgroup $H \subset G$ with $\Bbbk[G/H] = \Bbbk$ is said to be
\emph{epimorphic}. Various characterizations, properties, and
examples of epimorphic subgroups, as well as several conjectures and
open problems concerning them, can be found in~\cite{BBI},
\cite{BBII}, \cite{BBK}, and~\cite[\S\,23\,B]{Gr}.

It turns out that not every homogeneous space $G/H$ with epimorphic
$H$ admits a projective embedding with small boundary. A criterion
for this is given by Theorem~\ref{criterion} below. To formulate
this theorem, we need to recall some additional notions. A subgroup
$H \subset G$ is said to be \emph{observable} if $G/H$ is a
quasi-affine variety. An observable subgroup $H \subset G$ is said
to be a \emph{Grosshans subgroup} if the algebra $\Bbbk[G/H]$ is
finitely generated over~$\Bbbk$.

\begin{theorem}\label{criterion}
Let $H \subset G$ be an epimorphic subgroup. Then the following
conditions are equivalent:

\textup{(a)} $G/H$ admits a projective embedding with small
boundary;

\textup{(b)} there is a character $\chi$ of $H$ such that $\Ker
\chi$ is a Grosshans subgroup in $G$.
\end{theorem}

A complete proof of this theorem can be found
in~\cite[Theorem~1.1]{Ar}.

Under a certain assumption on $H$, a combinatorial classification of
all projective embeddings with small boundary for a given
homogeneous space $G/H$ is obtained in~\cite[\S\,3]{Ar}. Another
problem arising in connection with Theorem~\ref{criterion} is to
construct examples of epimorphic subgroups $H \subset G$ such that
the corresponding homogeneous spaces $G/H$ admit no projective
embeddings with small boundary. In view of condition~(b) of
Theorem~\ref{criterion}, such examples should be based on examples
of observable subgroups $H_0 \subset G$ such that the algebra
$\Bbbk[G/H_0]$ is not finitely generated over $\Bbbk$. In their
turn, examples of this kind are provided by linear counterexamples
to Hilbert's fourteenth problem.

We recall that, in general, Hilbert's fourteenth problem asks
whether for a subfield $L$ of the field $K(x_1, \ldots, x_n)$ of
rational functions in $n$ variables over a field $K$ such that $L
\supset K$ the algebra $L \cap K[x_1, \ldots, x_n]$ is finitely
generated over $K$. (Here $K[x_1, \ldots, x_n]$ is the algebra of
polynomials in $x_1, \ldots, x_n$.) An important special case of
this problem asks whether the algebra $K[x_1, \ldots, x_n]^G$ of
invariants of a linear action of a group $G$ on an $n$-dimensional
vector space is finitely generated over $K$. (This special case is
obtained from the general one with $L$ being the quotient field of
$K[x_1, \ldots, x_n]^G$.) Counterexamples to this special case of
Hilbert's fourteenth problem are called linear counterexamples.

At this moment, few counterexamples to Hilbert's fourteenth problem
are known. The first one, which turns out to be linear, was
discovered by Nagata in~1958~\cite{N}. In Nagata's example, a
13-dimensional unipotent group acts on a 32-dimensional vector
space. Much later, in 1997, Nagata's counterexample was considerably
simplified by Steinberg \cite{St} whose counterexample is now known
as Nagata-Steinberg's counterexample. In this example, the dimension
of the subgroup is reduced to~$6$ and that of the vector space to
$18$. The second counterexample, which is not linear, was
con\-structed in~1990 by Roberts~\cite{R} who used an approach
completely different from that of~Nagata. In 1994, based on Roberts'
counterexample, A'Campo-Neuen~\cite{A} constructed a linear
counterexample involving an action of a $12$-dimensional unipotent
group on a $19$-dimensional vector space. Subsequently, this
counterexample was followed by a series of linear counterexamples
(see, for instance,~\cite{Fr} and references therein). We note that
among linear counterexamples to Hilbert's fourteenth problem the key
role is played by counterexamples involving linear actions of
unipotent groups.

A natural way of obtaining examples of homogeneous spaces with
epimorphic stabilizer that admit no projective embeddings with small
boundary is as follows. First, one takes a linear counterexample to
Hilbert's fourteenth problem involving an action of a unipotent
group $H_0$ on a vector space $V$ over $\Bbbk$. Second, one fixes a
connected reductive group $G \subset \GL(V)$ containing $H_0$.
Third, one chooses an appropriate one-dimensional torus $S \subset
G$ normalizing $H_0$ such that the subgroup $H = S H_0$ is
epimorphic in~$G$. In this situation, the homogeneous space $G/H$
admits no projective embeddings with small boundary, see
Proposition~\ref{proposition} in Section~\ref{construction}.

The first example of a homogeneous space with epimorphic stabilizer
that admits no projective embeddings with small boundary was
mentioned in~\cite[7(b)]{BBII}. In this example, $G = \SL_2 \times
\ldots \times \SL_2$ ($16$ copies) and $H$ is obtained by extending
the group in Nagata's counter\-example by a one-dimensional torus.
An analogous example based on Nagata-Steinberg's counterexample was
considered (with complete proofs) in~\cite[\S\,2]{Ar}. In the
present paper, we construct a new example of that kind based on
A'Campo-Neuen's counter\-example. The precise formulations and
construction are given in Section~\ref{construction}.

In connection with these results the following question may be of
interest.

\begin{question*}
Let a unipotent group $H_0$ act linearly on a finite-dimensional
vector space $V$ and let $G$ be a connected reductive subgroup of\,
$\GL(V)$ containing $H_0$. Suppose that the algebra $\Bbbk[V]^{H_0}$
is not finitely generated over $\Bbbk$. Is there a one-dimensional
torus $S \subset G$ normalizing $H_0$ such that the group $H = S
H_0$ is epimorphic in~$G$?
\end{question*}

\section{Construction of the subgroup}\label{construction}

We put $G = \SL_{19}$ and denote by $V$ the space of the
tautological representation of $G$. We fix a basis $e_1, e_2,
\ldots, e_{19}$ in $V$. Further on, for any element of $G$, its
matrix is considered with respect to this basis.

Let $\mb G_a$ be the additive group of $\Bbbk$.

We consider a subgroup $H_0\simeq (\mb G_a)^{12}$ embedded in $G$ as
follows:
$$
\overline \mu = (\mu_0, \mu_1, \ldots, \mu_{11}) \mapsto h(\overline
\mu) =
\begin{pmatrix} E_4 & 0 \\ M(\overline \mu) & E_{15} \end{pmatrix},
$$
where $E_4$, $E_{15}$ are the identity matrices of order $4$, $15$,
respectively, and

\begin{equation} \label{M}
M(\overline \mu) =
\begin{pmatrix}
\mu_1 & \mu_0 & 0 & 0\\
\mu_2 & \mu_1 & 0 & 0\\
0 & \mu_2 & 0 & 0\\
\mu_3 & 0 & \mu_0 & 0\\
\mu_4 & 0 & \mu_3 & 0\\
0 & 0 & \mu_4 & 0\\
\mu_5 & 0 & 0 & \mu_0\\
\mu_6 & 0 & 0 & \mu_5\\
0 & 0 & 0 & \mu_6\\
\mu_7 & \mu_0 & 0 & 0\\
\mu_8 & \mu_7 & 0 & 0\\
\mu_9 & 0 & \mu_8 & 0\\
\mu_{10} & 0 & \mu_9 & 0\\
\mu_{11} & 0 & 0 & \mu_{10}\\
0 & 0 & 0 & \mu_{11}\\
\end{pmatrix}.
\end{equation}

The result of A'Campo-Neuen is as follows.

\begin{theorem}[\cite{A}]
The algebra $\Bbbk[V]^{H_0}$ is not finitely generated over~$\Bbbk$.
\end{theorem}

\begin{remark}
In~\cite{A} this theorem is proved for any ground field of
characteristic zero.
\end{remark}

To construct our example, we consider the one-dimensional torus
\begin{equation}\label{torus}
S = \lbrace\diag(s^{15}, s^{15}, s^{15}, s^{15}, \underbrace{s^{-4},
\ldots, s^{-4}}_{15})\mid s\in\Bbbk^\times\rbrace \subset G,
\end{equation}
where $\Bbbk^\times = \Bbbk \backslash \{0\}$ is the multiplicative
group of~$\Bbbk$. Clearly, $S$ normalizes $H_0$. We now put $H = S
H_0$. The main result of this paper is given by the following
theorem.

\begin{theorem}\label{main_theorem}
The subgroup $H$ is epimorphic in $G$.
\end{theorem}

This theorem will be proved in Section~\ref{main_theorem}.

\begin{corollary}
The homogeneous space $G/H$ admits no projective embeddings with
small boundary.
\end{corollary}

This corollary follows from Theorem~\ref{main_theorem} and the
following proposition.

\begin{proposition} \label{proposition}
Suppose we are given a linear action of a unipotent group $H_0$ on a
vector space $V$ over $\Bbbk$ such that the algebra $\Bbbk[V]^{H_0}$
is not finitely generated. Suppose that a connected reductive
subgroup $G \subset \GL(V)$ containing $H_0$ is fixed and there is a
one-dimensional torus $S \subset G$ normalizing $H_0$ such that the
group $H = S H_0$ is epimorphic in $G$. Then the homogeneous space
$G/H$ admits no projective embeddings with small boundary.
\end{proposition}

\begin{proof}
It suffices to show that condition~(b) of Theorem~\ref{criterion}
does not hold. This is done by the same argument as
in~\cite[Lemma~2.2]{Ar}.
\end{proof}

\section{Proof of Theorem~\ref{main_theorem}}

Before we prove Theorem~\ref{main_theorem}, let us fix some
notation.

We recall that $G = \SL_{19}$. Let $B$ (resp. $U$, $T$) be the Borel
subgroup (resp. the maximal unipotent subgroup, the maximal torus)
in $G$ consisting of all upper triangular (resp. upper
unitriangular, diagonal) matrices contained in $G$. Let $N_G(T)$
denote the normalizer of $T$ in~$G$, which consists of all monomial
matrices contained in~$G$. We denote by $\mf X(B)$ the weight
lattice of~$B$. The semigroup of dominant weights of $B$ is denoted
by $\mf X_+(B)$, $\mf X_+(B) \subset \mf X(B)$. For $i = 1, 2,
\ldots, 18$ we denote by $\pi_i$ the $i$-th fundamental weight of
$B$, which takes every upper triangular matrix to the product of its
first $i$ diagonal entries.

The simple $G$-module with highest weight $\lambda \in \mf X_+(B)$
is denoted by $V(\lambda)$, and its highest weight vector with
respect to $B$ is denoted by $v_\lambda$. Let $P_\lambda \subset G$
be the subgroup that stabilizes the line $\langle v_\lambda\rangle
\subset V(\lambda)$. This subgroup is a parabolic subgroup
containing the Borel subgroup~$B$. We identify the weight lattice
$\mf X(P_\lambda)$ of $P_\lambda$ with a sublattice of $\mf X(B)$ by
means of the natural embedding $B \hookrightarrow P_\lambda$.

Every dominant weight $\lambda$ of $B$ has the form $\lambda = a_1
\pi_1 + a_2 \pi_2 + \ldots + a_{18} \pi_{18}$ for some non-negative
integers $a_1, a_2, \ldots, a_{18}$. If $a_i > 0$ for some $i \in
\{1, 2, \ldots, 18\}$ then $P_\lambda$ stabilizes the line $\langle
v_{\pi_i}\rangle \subset V(\pi_i)$. At that, $P_\lambda$ acts on
$v_{\pi_i}$ by multiplication by the weight $\pi_i$. This weight
takes every matrix $A \in P_\lambda$ to the minor corresponding to
the first $i$ and last $i$ rows and columns of~$A$. (The lower left
$(19 - i) \times i$ block of $A$ consists of zero entries.)

In this section, we identify elements $s\in\Bbbk^\times$ and their
images in $S$, see~(\ref{torus}).

We now proceed to prove Theorem~\ref{main_theorem}.
By~\cite[Lemma~23.5]{Gr} it suffices to show that there are no
proper observable subgroups of $G$ containing~$H$. In view
of~\cite[Lemma~7.7]{Gr} the proof will be completed if we check the
following two conditions:

(1) for every non-trivial simple $G$-module $V(\lambda)$ and every
Borel subgroup $\widetilde B \subset G$ the highest weight vector of
$V(\lambda)$ with respect to $\widetilde B$ is not invariant under
$H$;

(2) there are no proper reductive subgroups of $G$ containing $H$.

Condition~(1) follows from Lemma~\ref{highest_weight_vectors}.
Condition~(2) will be checked using Lemma~\ref{reductive_subgroups}.
We now turn to formulate and prove the lemmas.

\begin{lemma}\label{highest_weight_vectors}
Let $\widetilde B\subset G$ be an arbitrary Borel subgroup and
$V(\lambda)$, $\lambda \ne 0$, an arbitrary simple $G$-module with
highest weight vector $\widetilde v_\lambda$ with respect to
$\widetilde B$. Then there is an element $h\in H$ such that $h \cdot
\widetilde v_\lambda \ne \widetilde v_\lambda$.
\end{lemma}

\begin{proof}
Assume that $h \cdot \widetilde v_\lambda = \widetilde v_\lambda$
for all $h \in H$. Since $\lambda \ne 0$, we have $\lambda = a_1
\pi_{i_1} + a_2 \pi_{i_2} + \ldots + a_m \pi_{i_m}$, where $1 \le m
\le 18$, $1 \le i_1 < i_2 < \ldots < i_m \le 18$, and $a_i
> 0$ for all $i = 1, \ldots, m$. The subsequent argument is divided
into several steps.

\emph{Step}~1. Since all Borel subgroups in $G$ are conjugated,
there exists an element $g_0\in G$ such that $\widetilde B = g_0
Bg_0^{-1}$. Then $\widetilde v_\lambda = \alpha g_0 \cdot v_\lambda$
for some $\alpha \ne 0$, whence
\begin{equation}\label{reformulation1}
g_0^{-1}hg_0\cdot v_\lambda = v_\lambda
\end{equation}
for all $h\in H$. Consider the Bruhat decomposition of $g_0$:
\begin{equation}\label{Bruhat}
g_0 = u \sigma b,
\end{equation}
where $u\in U$, $\sigma\in N_G(T)$, $b\in B$ are some fixed
elements. We may assume that $\sigma = \varepsilon \sigma_0$, where
$\sigma_0$ is a permutation matrix, $\varepsilon = 1$ for $\det
\sigma_0 = 1$, and $\varepsilon = e^{\pi \sqrt{-1}/19}$ for $\det
\sigma_0 = -1$. We now substitute expression~(\ref{Bruhat}) for
$g_0$ in~(\ref{reformulation1}). Since $b$ multiplies $v_\lambda$ by
a scalar, we have
\begin{equation}\label{reformulation2}
\sigma^{-1} u^{-1} hu \sigma \cdot v_\lambda = v_\lambda
\end{equation}
for all $h\in H$. Let $\tau \colon G \to G$ be the map given by the
formula $\tau(g) =\sigma^{-1}u^{-1}gu \sigma$. Taking into
account~(\ref{reformulation2}) we obtain $\tau(H) \subset
P_\lambda$.

\emph{Step}~2. Let $\nu$ be the permutation of the set $\{1, 2,
\ldots, 19\}$ that corresponds to $\sigma$. Then $\sigma(e_j) =
\varepsilon e_{\nu(j)}$ for $j = 1, \ldots, 19$. For each pair of
matrices $g = (g_{ij}) \in G$, $\overline g = \sigma^{-1} g \sigma$
we have $\overline g_{ij} = g_{\nu(i),\nu(j)}$ for $i,j = \{1,
\ldots, 19\}$. In particular, $\overline g_{jj} = g_{\nu(j),
\nu(j)}$ for $j = 1, \ldots, 19$. We note that under the map $g
\mapsto \overline g$ entries of $g$ lying in the same row (resp.
column) are transformed into elements of $\overline g$ that also lie
in the same row (resp. column).

\emph{Step}~3. Suppose $s \in S$. Then $u^{-1}su$ is an upper
triangular matrix whose diagonal entries are $s^{15}$, $s^{15}$,
$s^{15}$, $s^{15}$, $s^{-4}, \ldots, s^{-4}$. Further, the diagonal
entries of the matrix $\tau(s) = \sigma^{-1}u^{-1}su\sigma$ are
again $s^{15}$, $s^{15}$, $s^{15}$, $s^{15}$, $s^{-4}, \ldots,
s^{-4}$, perhaps in another order. At that, for every $i = 1,
\ldots, 18$ the determinant of the upper left $i \times i$ block of
$\tau(s)$ is equal to the product of all diagonal entries of this
block. Therefore, for every $j = 1, \ldots, m$ we have
$\pi_{i_j}(\tau(s)) = s^{b_j}$ for some $b_j \in \mb Z$. Moreover,
$b_j = 15k_j - 4l_j$, where
$$
k_j = \#\{k \in \{1, 2, \ldots,
i_j\}\mid \nu(k) \in \{1,2,3,4\}\}
$$
and
$$
l_j = \#\{k \in \{1, 2, \ldots, i_j\}\mid \nu(k) \notin
\{1,2,3,4\}\}.
$$
Clearly, $0 \le k_j \le 4$, $0 \le l_j \le 15$, and $k_j + l_j =
i_j$. The last equality implies that $(k_j, l_j) \notin \{(0, 0),
(4, 15)\}$, whence $b_j \ne 0$ for all $j = 1, \ldots, m$. Further,
the condition $\tau(s) \cdot v_\lambda = v_\lambda$ implies that
$\lambda(\tau(s)) = 1$ for all $s \in S$ and $a_1b_1 + a_2b_2 +
\ldots + a_mb_m = 0$. We conclude that there exists $j_0 \in \{1,
\ldots, m\}$ with $b_{j_0} > 0$. Put $i^* = i_{j_0}$, $k^* =
k_{j_0}$, and $l^* = l_{j_0}$. Since $b_{j_0} > 0$, we have $l^* <
15k^*/4$, in particular, $k^* > 0$. Obviously, for every matrix in
$\tau(H)$ its lower left $(19 - i^*) \times i^*$ block consists of
zero entries.

\emph{Step}~4. Suppose that $u = \begin{pmatrix} P & R \\ 0 & Q
\end{pmatrix}$ and $u^{-1} =
\begin{pmatrix} P^{-1} & R' \\ 0 & Q^{-1} \end{pmatrix}$, where $P$
and $Q$ are upper unitriangular matrices of order $4$ and $15$,
respectively, $R$ and $R'$ are $4 \times 15$ matrices, $R' =
-P^{-1}RQ^{-1}$. Let $h = h(\overline \mu) \in H_0$ be an arbitrary
element. Recall that $h = \begin{pmatrix} E_4 & 0 \\ M(\overline
\mu) & E_{15} \end{pmatrix}$ for some $\overline \mu \in
\Bbbk^{12}$, where $M(\overline \mu)$ is the matrix in~(\ref{M}).
Then
\begin{equation*}
u^{-1}hu =
\begin{pmatrix}
E_4 + R'M(\overline \mu)P & P^{-1}R + R'M(\overline \mu)R + R'Q\\
Q^{-1}M(\overline \mu)P & E_{15} + Q^{-1}M(\overline \mu)R
\end{pmatrix}.
\end{equation*}
We consider the $15 \times 4$ matrix $D = D(h) = Q^{-1}M(\overline
\mu)P$. Note that for every entry $d_{pq}$ of $D$ we have $d_{pq} =
m_{pq} + \sum c_{ij}m_{ij}$ where the sum is taken over all pairs
$(i, j)$ with $i \ge p$, $j \le q$, and $(i, j) \ne (p, q)$, the
coefficients $c_{ij}$ being uniquely determined by the matrix $u$.
Now, using the latter observation and the explicit form~(\ref{M}) of
the matrix $M(\overline \mu)$, we can successively choose elements
$\mu_{11}$, $\mu_{10}$, $\mu_9$, $\mu_8$, $\mu_7$, $\mu_0$, $\mu_6$,
$\mu_5$, $\mu_4$, $\mu_3$, $\mu_2$, $\mu_1 \in \Bbbk$ in such a way
that, for the corresponding element $h_0 \in H_0$, the submatrix
$D(h_0)$ of $u^{-1}h_0u$ has the form
\begin{equation}\label{matrix_D}
D(h_0) = \begin{pmatrix}
* & \diamond & \diamond & \diamond \\
* & * & \diamond & \diamond\\
\diamond & * & \diamond & \diamond \\
* & \diamond & \diamond & \diamond \\
* & \diamond & * & \diamond \\
\diamond & \diamond & * & \diamond \\
* & \diamond & \diamond & \diamond \\
* & \diamond & \diamond & * \\
\diamond & \diamond & \diamond & * \\
* & * & \diamond & \diamond \\
* & * & \diamond & \diamond \\
* & \diamond & * & \diamond \\
* & \diamond & * & \diamond \\
* & \diamond & \diamond & * \\
\diamond & \diamond & \diamond & *
\end{pmatrix},
\end{equation}
where the asterisks stand for non-zero entries and the diamonds
stand for the entries that are irrelevant for us.

\emph{Step}~5. We now turn to the element $h_0 \in H_0$ and the
corresponding matrix $D(h_0)$ found at the previous step. For $n =
1, 2, 3, 4$ we define numbers $Z(n)$ as follows. We consider all $15
\times n$ submatrices of $D(h_0)$. For each of them, we count the
number of non-zero rows. At last, we put $Z(n)$ to be the minimal
among the obtained values. Using the explicit form~(\ref{matrix_D})
of $D(h_0)$, we find that $Z(1) \ge 4$, $Z(2) \ge 8$, $Z(3) \ge 12$,
$Z(4) = 15$.

\emph{Step}~6. For $j = 1, \ldots, k^*$ we define (pairwise
distinct) numbers $n_1, \ldots, n_{k^*} \in \{1, \ldots, i^*\}$ by
the condition $\nu(n_j) = j$. The column $j$ of the matrix
$u^{-1}h_0u$ is obtained by applying the permutation $\nu$ to the
column $n_j$ of the matrix $\tau(h_0) =
\sigma^{-1}(u^{-1}h_0u)\sigma$ ($j = 1, \ldots, k^*$). Therefore,
none of the elements of $D(h_0)$ is such that its image under the
transformation $u^{-1}h_0u \mapsto \sigma^{-1}(u^{-1}h_0u)\sigma$ is
contained in one of the rows $n_1, \ldots, n_{k^*}$. Since the lower
left $(19 - i^*) \times i^*$ block of $\tau(h_0)$ is zero
(\emph{Step}~3), it follows that there is a $15 \times k^*$
submatrix of $D(h_0)$ whose number of non-zero rows is at most $i^*
- k^* = l^* < 15k^*/4$.

\emph{Step}~7. Comparing the results of the previous step with the
definition of the numbers $Z(n)$ (\emph{Step}~5) we get the
following inequality:
\begin{equation} \label{inequality}
Z(k^*) < 15k^*/4.
\end{equation}
Making use of the estimations of $Z(n)$ obtained at \emph{Step}~5,
we find that none of the possible values $k^* = 1, 2, 3, 4$
satisfies~(\ref{inequality}). This contradiction completes the proof
of the lemma.
\end{proof}

Thus condition~(1) is checked.

\begin{lemma}\label{reductive_subgroups}
Suppose that $F\subset G$ is a reductive subgroup containing $H$.
Then $F$ acts irreducibly on~$V$.
\end{lemma}

\begin{proof}
Since $F$ is reductive, its action on $V$ is completely reducible.
Therefore it suffices to show that $V$ contains no proper subspaces
invariant under~$F$.

Let $V_1\subset V$ be the subspace spanned by the vectors
$e_1,e_2,e_3,e_4$ and $V_2\subset V$ be the subspace spanned by the
vectors $e_5,e_6,\ldots, e_{19}$. Clearly, $V=V_1\oplus V_2$, $\dim
V_1=4$, and $\dim V_2=15$. We note that both of the subspaces
$V_1,V_2$ are invariant under the action of~$S$.

Suppose that $W\subset V$ is a subspace invariant under $F$ and
choose an arbitrary vector $w\in W$. Then $w=v_1+v_2$ for some
vectors $v_1\in V_1$ and $v_2\in V_2$. Acting on $w$ by the element
$(\sqrt{-1}, \sqrt{-1}, \sqrt{-1}, \sqrt{-1}, 1, 1, \ldots, 1)\in S
\subset H \subset F$, we obtain the vector $w' = \sqrt{-1}v_1+v_2$
that also lies in $W$. It follows that both vectors $v_1,v_2$ lie in
$W$. Therefore $W$ is the direct sum of its projections $W_1$ and
$W_2$ to the subspaces $V_1$ and $V_2$, respectively. Clearly, $W_1
= W \cap V_1$ and $W_2 = W \cap V_2$.

Let $w \in W_1$ be an arbitrary element. Then for every $h\in H_0$
we have $h \cdot w = w + v(w,h)$, where $v(w, h) \in W_2$. Regard
the set $H_0$ as a vector space. We define the map $\varphi_w \colon
H_0\to W_2$ by $\varphi_w(h)=v(w,h)$. In other words, for $w = a_1
e_1 + a_2 e_2 + a_3 e_3 + a_4 e_4$ and $\overline \mu \in
\Bbbk^{12}$ we have
\begin{equation}
\varphi_w(h(\overline \mu)) = M(\overline \mu)\begin{pmatrix}
a_1\\a_2\\a_3\\a_4\end{pmatrix}.
\end{equation}
Evidently, $\varphi_w$ is a linear map, therefore its image is a
subspace in $W_2$. Besides,
\begin{equation}
\dim\Ker\varphi_w + \dim\Im\varphi_w = \dim H_0 = 12.
\end{equation}
To find $\dim \Ker \varphi_w$ (and thereby $\dim \Im \varphi_w$), it
is sufficient to solve the linear system
\mbox{$\varphi_w(h(\overline \mu)) = 0$} in variables $\overline
\mu$. It is not hard to see that the dimension of the solution space
of this system depends only on the arrangement of non-zero
coordinates of $w$. The values of $\dim \Im \varphi_w$ for different
types of $w\in V_1$ are presented in Table~\ref{tabl}. In the first
row of this table $w$'s are written as column vectors, the asterisk
denotes a non-zero coordinate.{\sloppy

}

\renewcommand{\topfraction}{1}

\begin{table}[h]

\renewcommand{\tabcolsep}{0.3em}

\begin{center}
\caption{}\label{tabl}

\begin{tabular}{|c|c|c|c|c|c|c|c|c|c|c|c|c|c|c|c|c|}

\hline

$w$ &

{\begin{tabular}{c} $0$\\ $0$\\ $0$\\ $0$ \end{tabular}} &

{\begin{tabular}{c} $*$\\ $0$\\ $0$\\ $0$ \end{tabular}} &

{\begin{tabular}{c} $0$\\ $*$\\ $0$\\ $0$ \end{tabular}} &

{\begin{tabular}{c} $0$\\ $0$\\ $*$\\ $0$ \end{tabular}} &

{\begin{tabular}{c} $0$\\ $0$\\ $0$\\ $*$ \end{tabular}} &

{\begin{tabular}{c} $*$\\ $*$\\ $0$\\ $0$ \end{tabular}} &

{\begin{tabular}{c} $*$\\ $0$\\ $*$\\ $0$ \end{tabular}} &

{\begin{tabular}{c} $*$\\ $0$\\ $0$\\ $*$ \end{tabular}} &

{\begin{tabular}{c} $0$\\ $*$\\ $*$\\ $0$ \end{tabular}} &

{\begin{tabular}{c} $0$\\ $*$\\ $0$\\ $*$ \end{tabular}} &

{\begin{tabular}{c} $0$\\ $0$\\ $*$\\ $*$ \end{tabular}} &

{\begin{tabular}{c} $*$\\ $*$\\ $*$\\ $0$ \end{tabular}} &

{\begin{tabular}{c} $*$\\ $*$\\ $0$\\ $*$ \end{tabular}} &

{\begin{tabular}{c} $*$\\ $0$\\ $*$\\ $*$ \end{tabular}} &

{\begin{tabular}{c} $0$\\ $*$\\ $*$\\ $*$ \end{tabular}} &

{\begin{tabular}{c} $*$\\ $*$\\ $*$\\ $*$ \end{tabular}}\\

\hline

$\dim\Im\varphi_w$ & $0$ & $11$ & $4$ & $5$ & $5$ & $12$ & $12$ &
$12$ & $8$ & $8$ & $9$ & $12$ & $12$ & $12$ & $12$ & $12$\\

\hline

\end{tabular}

\end{center}

\end{table}

We now assume that $V = W \oplus W'$ for some proper subspaces $W,W'
\subset V$ invariant under $F$. Put $W_1 = W \cap V_1$, $W_2 = W
\cap V_2$, $W'_1 = W' \cap V_1$, and $W'_2 = W' \cap V_2$. It
follows from the above that $V_1 = W_1 \oplus W'_1$ and $V_2 = W_2
\oplus W'_2$. Without loss of generality we may assume that $\dim
W_1 \ge \dim W'_1$. Further we consider three possible cases.

\emph{Case}~1. $\dim W_1 = 2$, $\dim W'_1 = 2$. Then there are two
vectors $w_1 \in W_1$, $w'_1 \in W'_1$ such that each of them has at
least two non-zero coordinates. From Table~\ref{tabl} we find that
$\dim \Im \varphi_{w_1} \ge 8$ and $\dim \Im \varphi_{w'_1} \ge 8$.
On the other hand, $\Im \varphi_{w_1}, \Im \varphi_{w'_1} \subset
V_2$ and $\dim V_2=15$, whence the space $\Im \varphi_{w_1} \cap \Im
\varphi_{w'_1}$ has positive dimension. This is impossible because
$\Im \varphi_{w_1} \subset W_2$, $\Im \varphi_{w'_1} \subset W'_2$,
and $W_2 \cap W'_2 = \{0\}$.

\emph{Case}~2. $\dim W_1 = 3$, $\dim W'_1 = 1$. Then there are a
vector $w_1\in W_1$ with at least three non-zero coordinates and a
non-zero vector $w'_1\in W'_1$. From Table~\ref{tabl} we find that
$\dim \Im \varphi_{w_1} \ge 12$ и $\dim \Im \varphi_{w'_1} \ge 4$.
By the same reason as in \emph{Case}~1, the space $\Im \varphi_{w_1}
\cap \Im \varphi_{w'_1}$ has positive dimension, a contradiction.

\emph{Case}~3. $\dim W_1=4$, $\dim W'_1=0$. It is easy to see that
the space $\Im \varphi_{e_1}$ contains the vectors $e_5, e_6, e_8,
e_9, e_{11}, e_{12}, e_{14}, e_{15}, e_{16}, e_{17}, e_{18}$, the
space $\Im \varphi_{e_2}$ contains the vector $e_7$, the space $\Im
\varphi_{e_3}$ contains the vector $e_{10}$, and the space $\Im
\varphi_{e_4}$ contains the vectors $e_{13}, e_{19}$. Thus all the
basis vectors of $V$ lie in $W$, whence $W = V$ and $W' = 0$, a
contradiction.

In all the cases we have come to a contradiction, so the proof of
the lemma is completed.
\end{proof}

We now show that a reductive subgroup $F\subset G$ containing $H$
coincides with~$G$. First, we note that there are no non-trivial
bilinear forms on $V$ preserved by $F$ because this holds even for
$S$. Next, by Lemma~\ref{reductive_subgroups} the $F$-module $V$ is
simple. Therefore the center of $F$ is finite and $F$ is semisimple.
Moreover, $F$ is simple since otherwise the dimension of $V$ would
be a composite number, which is not the case (we have $\dim V=19$).
Obviously, the rank of $F$ is at least two. Further, $F$ contains
the unipotent subgroup $H_0$ of dimension $12$, whence $\dim F \ge 2
+ 2\cdot 12 = 26$.

Assume that $F \ne G$. Since there are no non-trivial bilinear forms
on $V$ preserved by $F$, it follows that $F$ can only be of type
$\SL_k$, $\Spin_{4l+2}$, or $\mathsf E_6$ (see~\cite[\S\,4.3]{OV}).
Further we consider these three cases separately. (In all the cases
below, our arguments rely on well-known facts from representation
theory of semisimple algebraic groups.)

1) $F$ is of type $\SL_k$. Clearly, $k \le 18$. Since $\dim F \ge
26$, we have $k \ge 6$. Every simple $\SL_k$-module $W$ with $\dim W
> k + 1$ has actually dimension at least $k(k-1)/2$, which is more
than $19$ for $k \ge 7$. It remains to consider the case $k = 6$.
Every simple $\SL_6$-module $W$ with $\dim W > 15$ has actually
dimension at least $20$.

2) $F$ is of type $\Spin_{4l + 2}$. Clearly, $l \le 4$. Since $\dim
F \ge 26$, we have $l \ge 2$. Every simple $\Spin_{18}$-module $W$
with $\dim W > 18$ has actually dimension at least $153$. Every
simple $\Spin_{14}$-module $W$ with $\dim W > 14$ has actually
dimension at least $64$. Every simple $\Spin_{10}$-module $W$ with
$\dim W > 16$ has actually dimension at least $45$.

3) $F$ is of type $\mathsf E_6$. Every non-trivial simple $\mathsf
E_6$-module has dimension at least $27$.

In all the cases we have obtained that $V$ is not a simple
$F$-module. This contradiction implies that $F = G = \SL_{19}$.

Thus, we have checked condition~(2). Theorem~\ref{main_theorem} is
proved.

\section*{Acknowledgements}

The author is grateful to I.\,V.~Arzhantsev for suggesting the
problem, his interest, and attention to the work.

\end{document}